\documentclass[12pt,dvipdfm]{article}
\usepackage{mathrsfs}
\usepackage{bbm}
\usepackage{amsfonts,mathrsfs,amscd,amssymb,amsthm,amsmath,bm,graphicx,psfrag,subfigure}
\usepackage{enumerate}

\large\normalsize

\newtheorem{thm}{Theorem}[section]
\newtheorem{defn}[thm]{Definition}
\newtheorem{prop}[thm]{Proposition}
\newtheorem{cor}[thm]{Corollary}
\newtheorem{lem}[thm]{Lemma}

\newtheorem{exmp}[thm]{Example}

\begin{document}

\title{Parking functions on  toppling matrices}
\author{
Jun Ma$^{a,}$\thanks{Email address of the corresponding author:
majun904@sjtu.edu.cn
\newline{\hspace*{1.8em}Partially supported by  SRFDP 20110073120068}}\\
 \and Yeong-Nan Yeh$^{b,}$\thanks{Partially supported by NSC 98-2115-M-001-010-MY3}
}\date{} \maketitle \vspace*{-1.2cm}\begin{center} \footnotesize
$^{a}$ Department of Mathematics, Shanghai Jiaotong University,
Shanghai\\
$^{b}$ Institute of Mathematics, Academia Sinica, Taipei\\

\end{center}

 \vspace*{-0.3cm}
\thispagestyle{empty}
\begin{abstract}
Let $\Delta$ be an integer $n \times n$-matrix  which satisfies the
 conditions:
$\det \Delta\neq 0$, $\Delta_{ij}\leq 0\text{ for }i\neq j,$ and
 there exists a vector ${\bf r}=(r_1,\ldots,r_n)>0$ such that
${\bf r}\Delta \geq 0$. Here the notation ${\bf r}> 0$ means that
$r_i>0$ for all $i$, and ${\bf r}\geq {\bf r}'$ means that $r_i\geq
r'_i$ for every $i$. Let $\mathscr{R}(\Delta)$ be the set of vectors
${\bf r}$ such that ${\bf r}>0$ and ${\bf r}\Delta\geq 0$. In this
paper, $(\Delta,{\bf r})$-parking functions are defined for any
${\bf r}\in\mathscr{R}(\Delta)$. It is proved that the set of
$(\Delta,{\bf r})$-parking functions is independent of ${\bf r}$ for
any ${\bf r}\in\mathscr{R}(\Delta)$. For this reason, $(\Delta,{\bf
r})$-parking functions  are simply called $\Delta$-parking
functions.  It is shown that the number of $\Delta$-parking
functions is less than or equal to  the determinant of $\Delta$.
Moreover,  the definition of $(\Delta,{\bf r})$-recurrent
configurations are given for any ${\bf r}\in\mathscr{R}(\Delta)$. It
is proved that the set of $(\Delta,{\bf r})$-recurrent
configurations is independent of ${\bf r}$ for any ${\bf
r}\in\mathscr{R}(\Delta)$. Hence,  $(\Delta,{\bf r})$-recurrent
configurations are simply called $\Delta$-recurrent configurations.
It is obtained that the number of $\Delta$-recurrent configurations
is larger than or equal to  the determinant of $\Delta$. A simple
bijection from $\Delta$-parking functions to $\Delta$-recurrent
configurations is established. It follows from this bijection that
the number of $\Delta$-parking functions and the number of
$\Delta$-recurrent configurations are both equal to the determinant
of $\Delta$.
\end{abstract}

\section{Introduction}
The classical parking functions are defined as follows.
There are $n$ parking spaces which are arranged in a line, numbered
$0$ to $n-1$ left to right and $n$ drivers labeled $1,\ldots,n$.
Each driver $i$ has an initial parking preference $a_i$. Drivers
enter the parking area in the order in which they are labeled. Each
driver proceeds to his preferred space and parks here if it is free,
or parks at the next unoccupied space to the right. If all the
drivers park successfully by this rule, then the sequence
$(a_1,\ldots,a_n)$ is called a parking function.

Konheim and Weiss \cite{konheim} introduced the conception of
 parking functions in the study of the linear probes of
random hashing function.  Riordan \cite{riordan} studied a relation
of parking problems to ballot problems. The most notable result
about parking functions is a bijection from the set of classical
parking functions of length $n$ to the set of labeled trees on $n+1$
vertices.

There are many generalizations of parking functions. Please refer to
\cite{cori,kung,pitman,stanley,yan1997,yan2001}. Postnikov and
Shapiro \cite{postnikov} introduced a new generalization, the
$G$-parking functions,  in the study of certain quotients of the
polynomial ring.
 Let $G$ be a
connected digraph with vertex set $V(G)=\{0,1,2,\ldots,n\}$ and edge
set $E(G)$. We allow $G$ to have multiple edges and loops. For any
$I\subseteq V(G)\setminus \{0\}$ and $v\in I$, define ${\rm
outdeg}_{I,G}(v)$ to be the number of edges from the vertex $v$ to a
vertex outside of the subset $I$ in $G$. $G$-parking functions are
defined as follows. \begin{itemize}\item A $G$-parking function is a
function $f:V(G)\setminus\{0\}\rightarrow \{0,1,2,\ldots\}$, such
that for every $I\subseteq V(G)\setminus\{0\}$ there exists a vertex
$v\in I$ such that $0\leq f(v)<{ outdeg}_{I,G}(v)$.
\end{itemize}
\noindent For the complete graph $G=K_{n+1}$ on $n+1$ vertices,
$K_{n+1}$-parking functions are exactly the classical parking
functions.

Chebikin and Pylyavskyy \cite{chebikin} gave a family of bijections
from the set of $G$-parking functions to the set of  the (oriented)
spanning trees of $G$. Let $L_G$ be the Laplace matrix that
corresponds to the connected digraph $G$ and $L_0$  the truncated
Laplace matrix obtained from the  matrix $L_G $  by deleting the
rows and columns indexed by $0$.  It follows from the matrix-tree
theorem that the number of $G$-parking functions is equal to $\det
L_0$.

One of the objective of the present paper is to generalize the
$G$-parking functions associated to an integer $n \times n$-matrix
$\Delta$ which satisfy the following
 conditions:
\begin{equation} \label{condition-matrix}  \begin{split}&\det
\Delta\neq 0;\\  &\Delta_{ij}\leq 0\text{ for }i\neq j;\\
&\text{there exists a vector }{\bf r}=(r_1,\ldots,r_n)>0\text{ such
that }{\bf r}\Delta \geq 0. \end{split}\end{equation} Here the
notation ${\bf r}> 0$ means that $r_i>0$ for all $i$ and ${\bf
r}\geq {\bf r}'$ means that $r_i\geq r'_i$ for every $i$.

 Let
$$\mathscr{R}(\Delta)=\{{\bf r}\in\mathbb{Z}^n\mid {\bf r}\Delta\geq
0\text{ and }{\bf r}>0\} $$ where $\mathbb{Z}$ is the set of
integers. For any ${\bf r}\in\mathscr{R}(\Delta)$, let
$${\bf c}={\bf c}({\bf r})=(c_1,\ldots,c_n)={\bf r}\Delta,
m=m({\bf r})=\sum\limits_{i=1}^nr_i.$$ Denote by $\Omega({\bf r})$
the set of integer vectors
$$\chi=(\chi(1),\cdots,\chi(n))$$ such that $$0\leq \chi(i)\leq
r_i\text{ for every }i\text{  and }\chi(i)\neq 0\text{ for some
}i.$$  Let $\Delta^j = (\Delta_{1j}, \ldots, \Delta_{nj})^T$ be the
$j$-th column of $\Delta$. There is a standard inner product
$\langle X,Y\rangle=\sum\limits_{i=1}^nX_iY_i$  on integer vectors
of length $n$. We define $(\Delta,{\bf r})$-parking functions as
follows:
\begin{defn}\label{definition} Let ${\bf r}\in\mathscr{R}(\Delta)$.  A $(\Delta,{\bf r})$-parking
function is a function $f:\{1,2,\ldots,n\}\rightarrow
\{0,1,2,\ldots\}$, such that for any $\chi\in \Omega({\bf r})$,
there exists a vertex $j$ with $\chi(j)\geq 1$ such that
$$0\leq f(j)<\langle\chi,\Delta^j\rangle.$$ Denote by $\mathcal {P}(\Delta,{\bf r})$ the set of $(\Delta,{\bf
r})$-parking functions.
\end{defn}
\begin{exmp} Let us consider a connected digraph $G$ with vertex set $\{0,1,\ldots,n\}$. The transposed matrix of the truncated
Laplace matrix $L_0$ of $G$ satisfies the conditions in
(\ref{condition-matrix}) and the vector ${\bf
1}\in\mathscr{R}(L_0^T)$, where the notation ${\bf 1}$ denotes a row
vector of length $n$ in which  all coordinate have value $1$.
$(L_0^T,{\bf 1})$-parking functions are exactly $G$-parking
functions.\end{exmp}
\begin{exmp}\label{exam-parking} The  matrix $\Delta$ and the vector ${\bf r}$ are given as follows:$$\Delta=\begin{pmatrix}2&-1\\-3&4\end{pmatrix}, {\bf
r}=(2,1).$$ Then $${\bf c}={\bf r}\Delta=(1,2),m=3,$$ and
$$\mathcal {P}(\Delta,{\bf r})=\{(0,0),(0,1),(0,2),(1,0),(1,1)\}.$$
\end{exmp}

A very interesting result is obtained: the set  of $(\Delta,{\bf
r})$-parking functions is independent of ${\bf r}$ for any ${\bf
r}\in\mathscr{R}(\Delta)$. For this reason,  $(\Delta,{\bf
r})$-parking functions are  simply called  $\Delta$-parking
functions and denote by $\mathcal{P}(\Delta)$ the set of
$\Delta$-parking functions.

Let
$\langle\Delta\rangle=\mathbb{Z}\Delta_1\oplus\mathbb{Z}\Delta_2\oplus\cdots\oplus\mathbb{Z}\Delta_n$
be the sublattice in $\mathbb{Z}^n$  spanned by the vectors
$\Delta_i$, where $\Delta_i = (\Delta_{i1}, \ldots, \Delta_{in})$ be
the $i$-th row of $\Delta$. We define an equivalence relation $\sim$
on $\mathbb{Z}^n$ by declaring that $f \sim f'$ if and only if $f -
f' \in \langle\Delta\rangle.$ It is proved that distinct
$\Delta$-parking functions cannot be equivalent. Thus, every
equivalent class of $\mathbb{Z}^n$ contains at most one
$\Delta$-parking function. Since  the order of the quotient of the
integer  lattice $\mathbb{Z}^n/\langle \Delta\rangle$ is $\det
\Delta$, it follows that the number of $\Delta$-parking functions is
less than or equal to $\det \Delta$.

Now we turn to the abelian sandpile model, also known as the
chip-firing game. It was introduced by Dhar \cite{dhar} and was
studied by many authors.  Gabrielov \cite{gabrielov2} introduced the
sandpile model for a class of toppling matrices, which   is more
general than in \cite{dhar}. We state this model as follows.

An integer $n \times n$-matrix $\Delta$ is a toppling matrix if it
satisfies the following conditions:
\begin{equation*}\label{topplingcondition}\begin{split}&\Delta_{ij}\leq 0\text{ for }i\neq j;\\&\text{ there exists a
vector }{\bf h}>0\text{ such that }\Delta {\bf
h}>0.\end{split}\end{equation*} These matrices is called
avalanche-finite redistribution matrices in \cite{gabrielov2}.

We list some properties of toppling matrices as follows.
\begin{prop}\label{laplace-matrix-topple-property}{\rm(Gabrielov, \cite{gabrielov2})} \begin{enumerate}[(1)]
\item A matrix $\Delta$ is a toppling matrix if and only if
its transposed matrix $\Delta^T$ is a toppling matrix.

\item If $\Delta$ is a toppling matrix, then all principal minors of $\Delta$ are strictly positive.

\item  Every integer matrix $\Delta$ such that
\begin{equation*}\label{toppling-graph-condition}\Delta_{ij}\leq 0\text{ for
}i\neq j;~~\sum\limits_{j=1}^n\Delta_{ij}\geq 0 \text{ for all
}i;\text{ and }det \Delta\neq 0.\end{equation*} is a toppling
matrix.
\end{enumerate}
\end{prop}

For a toppling matrix $\Delta$,  let $\Delta_i = (\Delta_{i1},
\ldots, \Delta_{in})$ be the $i$-th row of $\Delta$. A row vector
${\bf u} = (u_1,\ldots,u_n)$ is called a configuration if $u_i\geq
0$ for all $i$. For any vertex $i$, if $u_i \geq \Delta_{ii}$, we
say that the vertex $i$ is critical. A configuration $u$ is called
stable if no vertex is critical, i.e., $0 \leq u_i < \Delta_{ii}$
for all vertices $i$. A critical vertex $i$ is toppled, that is a
subtraction the vector $\Delta_i$ from the vector $u$. Furthermore,
a sequence of topplings is a sequence of vertices
$i_1,i_2,\ldots,i_k$ such that $i_j$ is a critical vertex of
$u-\Delta_{i_1}-\cdots -\Delta_{i_{j-1}}$ for any $1\leq j\leq k$.
   A representation vector for the sequence of topplings is a vector
   ${\bf r}=(r_1,\ldots,r_n)$
   with $$r_s=|\{j\mid i_j=s,1\leq j\leq k\}|.$$ Clearly,
   ${\bf u}-\sum\limits_{j=1}^k\Delta_{i_j}={\bf u}-{\bf r}\Delta$.

\begin{prop}\label{topple-into-stable}{\rm (Dhar, \cite{dhar})}Every configuration can be transformed into a stable configuration
by a sequence of topplings. This stable configuration does not
depend on the order in which topplings are performed.
\end{prop}

 For any $1\leq i\leq n$, the operator
$A_i$ is given by increasing $u_i$ by $1$, and then performing a
sequence of topplings that lead to a new stable configuration. So
the avalanche operators $A_1, \cdots,A_n $   map the set of stable
configurations to itself. Dhar \cite{dhar} proved that  the
avalanche operators $A_1, \ldots,A_n $ commute pairwise.

A stable configuration ${\bf u}$ is called recurrent if there are
positive integers $c_i$ such that $A^{c_i}_iu = u$ for all $i$.
Dhar \cite{dhar} also showed that the number of recurrent
configurations equals $ \det \Delta.$

A configuration ${\bf u}$ is allowed if there exists $j \in I$ such
that $$u_j\geq \sum\limits_{i\in I\setminus \{j\}}(-\Delta_{i,j})$$
for any nonempty subset $I$ of vertices. Dhar obtained a more
explicit characterization of  recurrent configurations:
 \begin{itemize}\item Every recurrent configuration is allowed.
\end{itemize}
Dhar suggested that  a configuration is recurrent if and only if it
is stable and allowed. Gabrielov \cite{gabrielov1} found that this
statement is not true in general, and proved the conjecture for a
toppling matrix $\Delta$ which has nonnegative column sums.  For
symmetric $\Delta = L_0$, Dhar's conjecture was proved in
\cite{cori2,ivashkevich,meester}, where $G$ is an undirected graph
and $L_0$ is the truncated Laplace matrix of $G$. Postnikov and
Shapiro \cite{postnikov} gave a bijection from $G$-parking functions
to recurrent configurations for the toppling matrix $\Delta=L_0$.

 Let $\Delta$ be an integer $n
\times n$-matrix  and satisfy the condition in
(\ref{condition-matrix}). Another objective of the present paper is
to show how $\Delta$-parking functions  are related to the sandpile
model. First, we show that an integer matrix $\Delta$  is a toppling
matrix if and only if it satisfies the conditions in
(\ref{condition-matrix}). Then for any ${\bf
r}\in\mathscr{R}(\Delta)$ we define $(\Delta,{\bf r})$-recurrent
configurations as follows.
\begin{defn}\label{r-recurrent} Let ${\bf u}$  be a configuration and ${\bf r}\in\mathscr{R}(\Delta)$. We say
that ${\bf u}$ is a $(\Delta,{\bf r})$-recurrent configuration if
${\bf u}$ is stable and the configuration ${\bf u}+{\bf r}\Delta$
can be transformed into
 ${\bf u}$
by a sequence of topplings. Denote by $\mathcal{R}(\Delta,{\bf r})$
the set of $(\Delta,{\bf r})$-recurrent configurations.
\end{defn}
\begin{exmp}\label{exam-recurrent} The  matrix $\Delta$ and the vector ${\bf r}$ are given as those in Example \ref{exam-parking}.
Then
$$\mathcal {R}(\Delta,{\bf r})=\{(1,3),(1,2),(1,1),(0,3),(0,2)\}.$$
\end{exmp}
Let $${\bf d}={\bf
d}(\Delta)=(\Delta_{11}-1,\Delta_{22}-1,\ldots,\Delta_{nn}-1).$$ For
any ${\bf r}\in\mathscr{R}(\Delta)$, we prove that a configuration
${\bf u}$ is a $(\Delta,{\bf r})$-recurrent configuration if and
only if ${\bf d}-{\bf u}$ is a $(\Delta,{\bf r})$-parking function.
This gives a bijection from $(\Delta,{\bf r})$-recurrent
configurations to $(\Delta,{\bf r})$-parking functions and implies
that the set of $(\Delta,{\bf r})$-recurrent configurations is
independent of ${\bf r}$ for any ${\bf r}\in\mathscr{R}(\Delta)$.
Hence, $(\Delta,{\bf r})$-recurrent configurations are  simply
called $\Delta$-recurrent configurations and denote by
$\mathcal{R}(\Delta)$ the set of $\Delta$-parking functions. We also
show that every equivalent class of $\mathbb{Z}^n$ contains at least
one $\Delta$-recurrent configuration. So  the number of
$\Delta$-recurrent configurations is larger than or equal to $\det
\Delta$. Combining the results about $\Delta$-parking functions, we
obtain the number of $\Delta$-parking functions and the number of
$\Delta$-recurrent configurations are both equal to $\det \Delta$.

Note that recurrent configurations for a toppling matrix $\Delta$
are exactly $\Delta$-recurrent configurations. Thus, with the
benefit of the bijection from $(\Delta,{\bf r})$-recurrent
configurations to $(\Delta,{\bf r})$-parking functions, we give
explicit characterization of recurrent configurations in the
sandpile model.

The rest of this paper is organized as follows.  In Section $2$, we
study $\Delta$-parking functions. In Section $3$, we study
$\Delta$-recurrent configurations.

\section{$\Delta$-parking functions}
In this section, we always let $\Delta= (\Delta_{ij})_{1\leq i,j\leq
n}$ be an integer $n \times n$-matrix and satisfy the
 conditions in  (\ref{condition-matrix}).  For any ${\bf r}\in\mathscr{R}(\Delta)$,
denote by  $\tilde{\Delta}=\tilde{\Delta}({\bf r})=(\tilde{\Delta}_{i,j})_{1\leq i,j\leq n}$ the following matrix $$\tilde{\Delta}=\tilde{\Delta}({\bf r})=\begin{pmatrix}r_1&0&\cdots&0\\
0&r_2&\ldots&0\\
\ldots&\ldots&\ldots&\ldots\\
0&0&\ldots&r_n\end{pmatrix}\Delta=\begin{pmatrix}r_1{\Delta}_{11}&r_1{\Delta}_{12}&\ldots&r_1{\Delta}_{1n}\\
r_2{\Delta}_{21}&r_2{\Delta}_{22}&\ldots&r_2{\Delta}_{2n}\\
\ldots&\ldots&\ldots&\ldots\\
r_n{\Delta}_{n1}&r_n{\Delta}_{n2}&\ldots&r_n{\Delta}_{nn}\end{pmatrix}.$$
Note that the column sums of $\tilde{\Delta}({\bf r})$ are
nonnegative since ${\bf r}\Delta\geq 0$, i.e.,
$\sum\limits_{i=1}^n\tilde{\Delta}_{i,j}\geq 0\text{ for every }j.$
We construct a digraph $D=D(\Delta,{\bf r})$ with vertex set $[0,n]$
as follows:
\begin{enumerate}[(a)]
\item For any $1\leq i,j\leq n$ and $i\neq j$,  we connect $i$ to $j$
by $-\tilde{\Delta}_{i,j}$ edges directed from $j$ to $i$; \item For
every $j\in\{1,2,\ldots,n\}$, we connect $j$ to $0$ by
$\sum\limits_{i=1}^n\tilde{\Delta}_{i,j}(\geq 0)$ edges directed
from $j$ to $0$.\end{enumerate}

Let ${\bf u} = (u_1,\ldots,u_n)$ is a vector of length $n$ and
$A=(a_{ij})$ an $n\times n$ matrix over a set $\{1,2,\ldots,n\}$.
For each nonempty subset $I\subseteq \{1,2,\ldots,n\}$, we denote by
$A[I]$ the submatrix of $A$ obtained by deleting the rows and
columns whose indices are in $\{1,2,\ldots,n\}\setminus I$ and by
${\bf u}[I]$ the vector obtained from ${\bf u}$ by deleting the
entries whose indices are in $\{1,2,\ldots,n\}\setminus I$.

\begin{lem}\label{det>=0}Let $\Delta= (\Delta_{ij})_{1\leq i,j\leq
n}$ be an integer $n \times n$-matrix and satisfy the  conditions:
$$\Delta_{ij}\leq 0\text{ for }i\neq j  \text{ and there exists a
row vector }{\bf r}>0\text{ such that }{\bf r}\Delta \geq 0.$$
 Then  the submatrix $\Delta[I]$ of $\Delta$ also satisfies the conditions above and $\det \Delta[I]\geq 0$  for each nonempty subset $I\subseteq
\{1,2,\ldots,n\}$.
\end{lem}
\begin{proof} Let ${\bf
r}=(r_1,\ldots,r_n)>0$ be a row vector  such that  ${\bf r}\Delta
\geq 0$. Fix a nonempty subset $I\subseteq \{1,2,\ldots,n\}$ and let
$J=\{1,2,\ldots,n\}\setminus I$. For any $i\in I$, we have $${\bf
r}[J]\Delta^i[J]\leq 0$$ where $\Delta^i$ is the $i$-th column of
$\Delta$. Hence, $${\bf r}[I]\Delta^i[I]={\bf r}\Delta^i-{\bf
r}[J]\Delta^i[J]\geq {\bf r}\Delta^i\geq 0 $$ and $${\bf
r}[I]\Delta[I]\geq 0$$

Let us consider the matrix $\tilde{\Delta}=\tilde{\Delta}({\bf
r})[I]$ and the graph $D=D(\Delta[I],{\bf r}[I])$. By the
matrix-tree theorem, the number of sink spanning trees rooted at $0$
in $D$ is
$$\det \tilde{\Delta}=\prod\limits_{i\in I}r_i\det \Delta[I].$$ Hence, we must
have
$$\det \Delta[I]\geq 0.$$
\end{proof}
\begin{prop}Let $\Delta= (\Delta_{ij})_{1\leq i,j\leq
n}$ be an integer $n \times n$-matrix. Then $\Delta$ satisfies the
conditions in  (\ref{condition-matrix}) if and only if  the matrix
$\Delta[I]$ satisfies the conditions in  (\ref{condition-matrix})
for each nonempty subset $I\subseteq \{1,2,\cdots,n\}$.
\end{prop}
\begin{proof}
Suppose that $\Delta$ satisfies the conditions in
(\ref{condition-matrix}). Let us consider the matrix
$\tilde{\Delta}=\tilde{\Delta}({\bf r})$ and the graph
$D=D(\Delta,{\bf r})$. By the matrix-tree theorem, the number of
sink spanning trees rooted at $0$ in $D$ is $$\det
\tilde{\Delta}=r_1 \cdots r_n\det \Delta.$$ Hence, we must have
$$\det \Delta>0$$ since $\det \Delta\neq 0$. This also implies that the number of sink spanning
forests rooted at $J$ in $D$ is larger than $0$, where
${J}=\{1,2,\ldots,n\}\setminus I$. By the matrix-forest theorem, we
have
$$\det \tilde{\Delta}[I]=\prod\limits_{i\in I}r_i\det
\Delta[I]>0\text{ and }\det \Delta[I]>0.$$ Hence, the matrix
$\Delta[I]$ satisfy the conditions in  (\ref{condition-matrix}) by
Lemma \ref{det>=0}.
\end{proof}
\begin{cor} Let $\Delta= (\Delta_{ij})_{1\leq i,j\leq
n}$ be an integer $n \times n$-matrix and satisfy the conditions in
 (\ref{condition-matrix}). Then $$\det \Delta[I]>0$$  for each nonempty subset $I\subseteq
\{1,2,\ldots,n\}$.
\end{cor}

Denote by $V({\bf r})$ a multiset with $r_i$ copies of $i$ for every
$i\in\{1,2,\cdots,n\}$. For any $\chi\in\Omega({\bf r})$ we can
obtain a submultiset of $V({\bf r})$ by giving $\chi(i)$ copies of
$i$ for every $i\in\{1,2,\cdots,n\}$. Thus, we call $\chi$ the
characteristic function of $W$. Conversely, for any submultiset $W$
of $V({\bf r})$, let
 $\chi(i)$ be the occurrence number of sites $i$ in $W$ for every
 $i\in\{1,2,\cdots,n\}$. Then $\chi\in \Omega({\bf
r})$.


\begin{lem}\label{parking-sequence}For any ${\bf r}\in\mathscr{R}(\Delta)$, let $m=m({\bf r})=\sum\limits_{i=1}^nr_i$. Then $f$ is a  $(\Delta,{\bf r})$-parking
function  if and only if there is a sequence of vertices in the
multiset $V({\bf r})$
$$\pi(1),\ldots,\pi(m)$$ such that for every $i\in\{1,2,\cdots,m\}$, $$0\leq f(\pi(i))<\langle\chi_i,\Delta^{\pi(i)}\rangle$$
where   $\chi_i$ is the characteristic function of the multiset
$\{\pi(i),\pi(i+1),\ldots,\pi(m)\}.$
\end{lem}
\begin{proof}Suppose that  $f$ is a  $(\Delta,{\bf r})$-parking
function. We construct a sequence $$\pi(1),\pi(2),\ldots,\pi(m)$$ of
vertices in $V({\bf r})$
by the following algorithm.\\
{\bf Algorithm A.}
\begin{itemize}\item Step 1. Let $W_1=V({\bf r})$, $\chi_1$ the
characteristic function of $W_1$ and $$U_1=\left\{j\in W_1\mid 0\leq
f(j)<\langle\chi_1,\Delta_j\rangle\right\}.$$ Set $\pi(1)\in U_1$.
\item Step 2. At time $i\geq 2$, suppose $\pi(1),\ldots,\pi(i-1)$
are determined. Let $$W_i=V({\bf r})\setminus\{\pi(j)\mid
j=1,\ldots,i-1\},$$ $\chi_i$ the characteristic function of $W_i$
and
$$U_i=\left\{j\in W_i\mid 0\leq
f(j)<\langle\chi_i,\Delta_j\rangle\right\}.$$ Set $\pi(i)\in U_i$.
\end{itemize} By Algorithm A, iterating Step $2$ until $i=m$, we obtain the sequence of vertices as
desired.

Conversely, suppose that there is a sequence of vertices in $V({\bf
r})$
$$\pi(1),\ldots,\pi(m)$$ such that for every $i\in\{1,2,\ldots,m\}$,  $$0\leq f(\pi(i))<\langle\chi_i,\Delta^{\pi(i)}\rangle=\chi_i(\pi(i))\Delta_{\pi(i),\pi(i)}
-\sum\limits_{j\neq \pi(i)}\chi_i(j)(-\Delta_{j,\pi(i)}),$$ where
$\chi_i$ is the characteristic function of
$\{\pi(i),\pi(i+1),\ldots,\pi(m)\}$.

Let $\chi\in \Omega({\bf r})$ and $U= \{j\mid \chi(j)\neq 0\}$. For
every $j\in U$, use $\theta(j)$ to denote the  unique index
$k\in\{1,2,\ldots,m\}$ such that $\pi(k)=j$ and $
\chi_k(j)=\chi(j)$. Let $i=\min\{\theta(j)\mid j\in U\}$. Then
$\chi_i(\pi(i))=\chi(\pi(i))\neq 0$ and $\chi_i(k)\geq \chi(k)$ for
every $k\neq \pi(i)$. Thus
\begin{eqnarray*}f(\pi(i))&<&\chi_i(\pi(i))\Delta_{\pi(i),\pi(i)}
-\sum\limits_{j\neq \pi(i)}\chi_i(j)(-\Delta_{j,\pi(i)})\\
&= &\chi(\pi(i))\Delta_{\pi(i),\pi(i)}
-\sum\limits_{j\neq \pi(i)}\chi_i(j)(-\Delta_{j,\pi(i)})\\
&\leq &\chi(\pi(i))\Delta_{\pi(i),\pi(i)}
-\sum\limits_{j\neq \pi(i)}\chi(j)(-\Delta_{j,\pi(i)})\\
&=&\langle\chi,\Delta^{\pi(i)}\rangle.\end{eqnarray*} By Definition
\ref{definition}, $f$ is a $(\Delta,{\bf r})$-parking function.
\end{proof}

\begin{lem}\label{parking-r-r'} Suppose that ${\bf r},{\bf r}'\in\mathscr{R}(\Delta)$ and ${\bf r}\leq {\bf
r}'$. Then $\mathcal {P}(\Delta,{\bf r}')\subseteq \mathcal
{P}(\Delta,{\bf r}).$
\end{lem}
\begin{proof} Note that $\Omega({\bf r})\subseteq\Omega({\bf
r}')$ since ${\bf r}\leq {\bf r}'$. So $f$ is a $(\Delta,{\bf
r})$-parking function if it is a $(\Delta,{\bf r}')$-parking
function. Hence we have $\mathcal {P}(\Delta,{\bf r}')\subseteq
\mathcal {P}(\Delta,{\bf r}).$
\end{proof}
\begin{lem} Suppose that ${\bf r},{\bf r}'\in\mathscr{R}(\Delta)$. Then $$\mathcal {P}(\Delta,{\bf r}+{\bf r}')=\mathcal
{P}(\Delta,{\bf r})\cap\mathcal {P}(\Delta,{\bf r}').$$
\end{lem}
\begin{proof} By Lemma \ref{parking-r-r'}, we have $\mathcal {P}(\Delta,{\bf r}+{\bf r}')\subseteq\mathcal {P}(\Delta,{\bf
r})$ and $\mathcal {P}(\Delta,{\bf r}+{\bf r}')\subseteq\mathcal
{P}(\Delta,{\bf r}')$. So, $\mathcal {P}(\Delta,{\bf r}+{\bf
r}')\subseteq \mathcal {P}(\Delta,{\bf r})\cap\mathcal
{P}(\Delta,{\bf r}').$

Conversely, let $m=\sum\limits_{i=1}^nr_i$ and
$m'=\sum\limits_{i=1}^nr'_i$. For any $f\in \mathcal {P}(\Delta,{\bf
r})\cap\mathcal {P}(\Delta,{\bf r}')$, by Lemma
\ref{parking-sequence}, there are a sequence
$$\pi(1),\ldots,\pi(m)$$ of vertices in $V({\bf r})$ such that  $$0\leq
f(\pi(i))<\langle\chi_i,\Delta^{\pi(i)}\rangle$$ for every
$i\in\{1,2,\cdots,m\}$  and a sequence
$$\pi'(1),\ldots,\pi'(m')$$ of vertices in $V({\bf r}')$ such that $$0\leq
f(\pi'(i))<\langle\chi'_i,\Delta^{\pi'(i)}\rangle$$  for every
$i\in\{1,2,\cdots,m'\}$ where  $\chi_i$ and $\chi'_i$ are the
characteristic functions of $\{\pi(i),\pi(i+1),\ldots,\pi(m)\}$ and
$\{\pi'(i),\pi'(i+1),\ldots,\pi'(m')$ respectively.

Let us consider the following sequence
$$\sigma(1),\ldots,\sigma(m),\sigma(m+1),\ldots,\sigma(m+m')$$
where
$$\sigma(i)=\left\{\begin{array}{lll}\pi(i)&\text{ if }&1\leq i\leq m\\
\pi'(i-m)&\text{ if }&1+m\leq i\leq m+m'\end{array}\right.$$ For
every $i=1,2,\cdots,m+m'$, let $\hat{\chi}_i$ is the characteristic
functions of $\{\sigma(i),\cdots,\sigma(m+m')\}$. Then we have
$$f(\sigma(i))=\left\{\begin{array}{lll}f(\pi(i))&\text{ if }&1\leq i\leq m\\
f(\pi'(i-m))&\text{ if }&1+m\leq i\leq m+m'\end{array}\right.$$ and
$$\langle \hat{\chi}_i,\Delta^{\sigma(i)}\rangle=\left\{\begin{array}{lll}\langle {\chi}_i+{\bf r'},\Delta^{\sigma(i)}\rangle=\langle {\chi}_i,\Delta^{\sigma(i)}\rangle+\langle {\bf r'},\Delta^{\sigma(i)}\rangle&\text{ if }&1\leq i\leq m\\
\langle {\chi'}_i,\Delta^{\sigma(i)}\rangle&\text{ if }&1+m\leq
i\leq m+m'\end{array}\right.$$ Since ${\bf r}'\Delta\geq 0$, we have
$$f(\sigma(i))<\langle \hat{\chi}_i,\Delta^{\sigma(i)}\rangle$$ for every
$i=1,2,\ldots,m+m'$. By Lemma \ref{parking-sequence}, $f$ is a
$(\Delta,{\bf r}+{\bf r}')$-parking function. Hence, $\mathcal
{P}(\Delta,{\bf r}+{\bf r}')=\mathcal {P}(\Delta,{\bf
r})\cap\mathcal {P}(\Delta,{\bf r}').$
\end{proof}
\begin{cor}\label{parking-br-ar+br}\begin{enumerate}[(1)]\item Suppose that ${\bf r}\in\mathscr{R}(\Delta)$ and $b$ is
a positive integer. Then
$$\mathcal {P}(\Delta,b{\bf r})=\mathcal {P}(\Delta,{\bf r}).$$ \item
Suppose that ${\bf r}_1,{\bf r}_2,\cdots,{\bf
r}_k\in\mathscr{R}(\Delta)$ and $b_1,b_2,\cdots, b_k$ are $k$
positive integers. Then
$$\mathcal {P}(\Delta,b_1{\bf r}_1+b_2{\bf r}_2+\cdots+b_k{\bf
r}_k)=\bigcap\limits_{i=1}^k\mathcal {P}(\Delta,{\bf
r}_i).$$\end{enumerate}
\end{cor}

\begin{thm}\label{thm-r-parking-set-same} For any ${\bf r},{\bf r}'\in\mathscr{R}(\Delta)$, $\mathcal{P}(\Delta,{\bf r})= \mathcal{P}(\Delta,{\bf
r}').$
\end{thm}
\begin{proof} Note that there is a positive $b$ such that $b{\bf r}\geq {\bf
r}'$ since ${\bf r}>0$. By Lemma \ref{parking-r-r'} and Corollary
\ref{parking-br-ar+br}(1), we have
$$\mathcal{P}(\Delta,{\bf r})=\mathcal{P}(\Delta,b{\bf r})\subseteq \mathcal{P}(\Delta,{\bf r}').$$
Similarly, we have $\mathcal{P}(\Delta,{\bf r}')\subseteq
\mathcal{P}(\Delta,{\bf r}).$ Hence, $\mathcal{P}(\Delta,{\bf r})=
\mathcal{P}(\Delta,{\bf r}')$.
\end{proof}

Theorem \ref{thm-r-parking-set-same} tells us that the set of
$(\Delta,{\bf r})$-parking functions is independent of ${\bf r}$ for
any ${\bf r}\in\mathscr{R}(\Delta)$. So, $(\Delta,{\bf r})$-parking
functions  are simply called $\Delta$-parking functions.

\begin{lem}\label{f-f'-parking-coset} Let ${\bf r}\in\mathscr{R}(\Delta)$. Suppose $f$ and $f'$  are two $(\Delta,{\bf
r})$-parking functions. If $f'-f\in \langle \Delta\rangle$, then
$f'=f$.
\end{lem}
\begin{proof} Assume that $f'\neq f$. Then $f'-f=x\Delta$ and $x\neq
0$. By symmetry, we may suppose that $x_j>0$ for some $j\in
\{1,2,\ldots,n\}$.

Let $b$ be a positive integer such that $$\min\{br_i\mid
i=1,2,\cdots,n\}\geq \max \{x_j\mid x_j>0\text{ and }1\leq j\leq
n\}.$$  Let
$$\chi(j)=\left\{\begin{array}{lll}x_j&\text{ if }& x_j>0\\
0&\text{ if }& x_j\leq 0\end{array}\right.$$ for each
$j=1,2,\cdots,n$. Then $\chi\in\Omega(\Delta,b{\bf r})$ and for any
$j$ with $\chi(j)>0$
\begin{eqnarray*}0\leq f(j)&=&f'(j)-\sum\limits_{k=1}^nx_k\Delta_{k,j}\\
&\leq &f'(j)-\sum\limits_{k=1\atop
x_k>0}^nx_k\Delta_{k,j}\\
&=&f'(j)-\langle \chi,\Delta^j\rangle.\end{eqnarray*} So, for any
$j$ with $\chi(j)>0$, we have $f'(j)\geq \langle
\chi,\Delta^j\rangle$. Hence $f'$ is not a $(\Delta,b{\bf
r})$-parking function since $\chi\in\Omega(b{\bf r})$. But Corollary
\ref{parking-br-ar+br}(1) implies that  $f'$ is a $(\Delta,b{\bf
r})$-parking functions since  $f'$ is a $(\Delta,{\bf r})$-parking
functions, a contradiction.
\end{proof}

Lemma \ref{f-f'-parking-coset} implies distinct $\Delta$-parking
functions cannot be equivalent and  every equivalent class of
$\mathbb{Z}^n$ contains at most one $\Delta$-parking function. So we
obtain the following corollary.

\begin{cor}\label{parking-<=det} The number of $\Delta$-parking functions is less than or equal to $ \det
\Delta$.
\end{cor}
\begin{proof} Since  the order of the quotient of
the integer  lattice $\mathbb{Z}^n/\langle \Delta\rangle$ is $\det
\Delta$, it follows from Lemma \ref{f-f'-parking-coset} and Theorem
\ref{thm-r-parking-set-same} that $|\mathcal{P}(\Delta)|\leq \det
\Delta$.
\end{proof}
\section{$\Delta$-recurrent configurations}

In this section, we shall give the definition of $\Delta$-recurrent
configurations and study their properties, where the matrix $\Delta$
satisfies the conditions in  (\ref{condition-matrix}).

\begin{prop}A  matrix $\Delta$  is a toppling matrix if and only if it satisfies the conditions in  (\ref{condition-matrix}).
\end{prop}
\begin{proof} Suppose that $\Delta$ satisfies the conditions in  (\ref{condition-matrix}). Then there exists a vector $${\bf r}=(r_1,\cdots,r_n)>0$$ such
that ${\bf r}\Delta \geq 0.$ Let $$\tilde{\Delta}=\begin{pmatrix}r_1&0&\cdots&0\\
0&r_2&\cdots&0\\
\cdots&\cdots&\cdots&\cdots\\
0&0&\cdots&r_n\end{pmatrix}\Delta=\begin{pmatrix}r_1\Delta_{11}&r_1\Delta_{12}&\cdots&r_1\Delta_{1n}\\
r_2\Delta_{21}&r_2\Delta_{22}&\cdots&r_2\Delta_{2n}\\
\cdots&\cdots&\cdots&\cdots\\
r_n\Delta_{n1}&r_n\Delta_{n2}&\cdots&r_n\Delta_{nn}\end{pmatrix}$$
By Proposition \ref{laplace-matrix-topple-property}(1) and (3),
$\tilde{\Delta}$ is a toppling matrix. There exists a column vector
${\bf h}>0$ such that $\tilde{\Delta}{\bf h}>0$. Suppose ${\bf
v}=(v_1,v_2,\ldots,v_n)^T=\Delta {\bf h}$. We have
$$\tilde{\Delta}{\bf h}=(r_1v_1,r_2v_2,\ldots,r_nv_n)^T>0.$$ This
implies $\Delta {\bf h}>0$ and $\Delta$  is a toppling matrix.

Conversely, suppose that $\Delta$  is a toppling matrix. By
Proposition \ref{laplace-matrix-topple-property}(2), we have $\det
\Delta\neq 0$. By Proposition
\ref{laplace-matrix-topple-property}(1), its transposed matrix
$\Delta^T$ is a toppling matrix. There exists a column vector ${\bf
h}>0$ such that $\Delta^T{\bf h}>0$. So ${\bf h}^T\Delta>0$ and
$\Delta$ satisfies the conditions in $(1)$.
\end{proof}

\begin{prop}\label{topple-matrix-property}
  A matrix $\Delta$ is a toppling matrix if and only if  all principal
minors of $\Delta$ are toppling matrices.
\end{prop}
\begin{proof}
Let $\Delta$ be a toppling matrix and ${\bf h} = (h_1,\cdots ,h_n)>
0$ an integer vector such that $\Delta {\bf h}^T > 0$. For each
nonempty subset $I\subseteq \{1,2,\cdots,n\}$, let
$J=\{1,2,\cdots,n\}\setminus I$ and suppose that $\Delta_i$ is the
$i$-th row of $\Delta$. Then for any $i\in I$, $\Delta_i[J]{\bf
h}[J]^T\leq 0$ and
$$\Delta[I]_i{\bf h}[I]^T=\Delta_i[I]{\bf h}[I]^T=\Delta_i
{\bf h}^T-\Delta_i[J]{\bf h}[J]^T\geq \Delta_i {\bf h}^T >0.$$ This
implies that $\Delta[I]{\bf h}[I]^T>0$ and $\Delta[I]$ is a toppling
matrix.
\end{proof}

\begin{prop} Let $\Delta= (\Delta_{ij})_{1\leq i,j\leq n}$ be an integer $n
\times n$-matrix with $\Delta_{ij}\leq 0\text{ for }i\neq j$ and
$adj(\Delta)=(A_{ij})_{1\leq i,j\leq n}$  the adjugate of $\Delta$.
Then $\Delta$ is a toppling matrix if and only if $\det \Delta>0$,
$A_{ii}> 0$ and $A_{ij}\geq 0$ for any $i\neq j$.
\end{prop}
\begin{proof} Suppose that $\det \Delta>0$,
$A_{ii}> 0$ and $A_{ij}\geq 0$ for any $i\neq j$. Let ${\bf
h}=adj(\Delta){\bf 1}^T$. Then $${\bf h}>0\text{ and }\Delta {\bf
h}=\Delta adj(\Delta){\bf 1}^T=(\det \Delta) {\bf 1}^T>0.$$  Hence,
$\Delta$ is a toppling matrix.

Conversely, suppose $\Delta$ is a toppling matrix. Proposition
\ref{laplace-matrix-topple-property}(2) implies $\det \Delta>0$ and
there is a recurrent configuration ${\bf u}$. It follows from the
definition of recurrent configurations that for every $i$ there is a
positive integer  $c_i$ such that $A^{c_i}_i{\bf u} = {\bf u}$.
 This means that the configuration ${\bf u}+c_i{\bf e}_i$ can be transformed into
 ${\bf u}$
by a sequence of topplings, where ${\bf e}_i$ is a row vector of
length $n$ in which  the $i$-th coordinate has value $1$ and the
other coordinate has value $0$. Suppose that representation vector
for the sequence of topplings is
   $${\bf r}_i=(r_{i1},\ldots,r_{in}).$$ Then ${ r}_{ii}\geq 1, { r}_{ij}\geq 0$ for $i\neq j$ and $${\bf r}_i\Delta=c_i{\bf
   e}_i.$$ So $${\bf r}_i=\frac{c_i}{\det \Delta}{\bf
   e}_i adj(\Delta)=\frac{c_i}{\det
   \Delta}(A_{i1},\cdots,A_{in}).$$ Hence, $A_{ii}> 0$ and $A_{ij}\geq 0$ for any $i\neq j$.
\end{proof}

Now, we always let $\Delta$ be an $n\times n$ integer matrix
satisfying the condition in  (\ref{condition-matrix}).

\begin{lem}\label{lem-parking-recurrent-same} For any ${\bf r}\in\mathscr{R}(\Delta)$,  a configuration ${\bf u}$ is a $(\Delta,{\bf r})$-recurrent configuration if
and only if ${\bf d}-{\bf u}$ is a $(\Delta,{\bf r})$-parking
function.
\end{lem}
\begin{proof} Let $m=m({\bf
r})=\sum\limits_{j=1}^nr_j$. Suppose that ${\bf u}$ is a
$(\Delta,{\bf r})$-recurrent configuration. By Definition
\ref{r-recurrent}, the configuration $u+{\bf r}\Delta$ can be
transformed into
 $u$
by a sequence $i_1,i_2,\ldots,i_m$ of topplings. Note that ${\bf r}$
is the representation vector for the sequence $i_1,i_2,\ldots,i_m$.
For every $j\in\{1,2,\ldots,m\}$, let  $\chi_j$ be the
characteristic function of the multiset
$\{i_j,i_{j+1},\ldots,i_m\}$. Then we have
$$u_{i_j}+\sum\limits_{k=j}^m\Delta_{i_k,i_j}\geq \Delta_{i_j,i_j}$$
and $$({\bf d}-{\bf u})_{i_j}=\Delta_{i_j,i_j}-1-u_{i_j}\leq
\sum\limits_{k=j}^m\Delta_{i_k,i_j}-1=\langle\chi_j,\Delta^{i_j}\rangle-1<\langle\chi_j,\Delta^{i_j}\rangle.$$
It follows from Lemma \ref{parking-sequence} that ${\bf d}-{\bf u}$
is a $(\Delta,{\bf r})$-parking function.

Conversely, suppose $f={\bf d}-{\bf u}$ is a $(\Delta,{\bf
r})$-parking function. By Proposition \ref{parking-sequence}, there
is a sequence of vertices in $V( {\bf r})$
$$\pi(1),\ldots,\pi(m)$$ such that for every $i\in\{1,2,\ldots,m\}$  $$0\leq f(\pi(i))<\langle\chi_i,\Delta^{\pi(i)}\rangle$$
where   $\chi_i$ is the characteristic function of
$\{\pi(i),\pi(i+1),\ldots,\pi(m)\}.$ So,
\begin{eqnarray*}u_{\pi(i)}&=&\Delta_{\pi(i),\pi(i)}-1-f(\pi(i))\\
&>&\Delta_{\pi(i),\pi(i)}-1-\langle\chi_i,\Delta^{\pi(i)}\rangle\\
&=&\Delta_{\pi(i),\pi(i)}-1-\sum\limits_{k=i}^m\Delta_{\pi(k),\pi(i)}\end{eqnarray*}
and $$u_{\pi(i)}+\sum\limits_{k=i}^m\Delta_{\pi(k),\pi(i)}\geq
\Delta_{\pi(i),\pi(i)}.$$ This implies that  ${\bf u}+{\bf r}\Delta$
can be transformed into
 ${\bf u}$
by the sequence $\pi(1),\pi(2),\ldots,\pi(m)$ of topplings.
\end{proof}

\begin{thm}\label{thm-r-recurrent-set-same}For any ${\bf r}\in\mathscr{R}(\Delta)$,
$\mathcal{R}(\Delta,{\bf r})= \mathcal{R}(\Delta,{\bf r}')$.
\end{thm}
\begin{proof} The required results follows from Lemma
\ref{lem-parking-recurrent-same} and Theorem
\ref{thm-r-parking-set-same}.
\end{proof}

Theorem \ref{thm-r-recurrent-set-same} tells us that the set of
$(\Delta,{\bf r})$-recurrent configurations is independent of ${\bf
r}$ for any ${\bf r}\in\mathscr{R}(\Delta)$. So, $(\Delta,{\bf
r})$-parking functions are simply called $\Delta$-recurrent
configurations.

\begin{lem}\label{r-recurrent-coset}Let ${\bf r}\in\mathscr{R}(\Delta)$. For any  integer vector ${\bf v} = (v_1, \ldots, v_n)$, there
exists a $(\Delta,{\bf r})$-recurrent configuration ${\bf u}$ such
that ${\bf v}-{\bf u}\in \langle\Delta\rangle$.
\end{lem}
\begin{proof} Note that $\det \Delta>0$  and $(\det \Delta){\bf 1}=({\bf 1}adj(\Delta))\Delta\in\langle\Delta\rangle$. For any
integer vector ${\bf v} = (v_1, \ldots, v_n)$, there exists a
positive integer $k$ such that ${\bf v}+k(\det \Delta){\bf 1}>0$. It
is sufficient to prove for any  configuration ${\bf v} = (v_1,
\ldots, v_n)$, there exists a $(\Delta,{\bf r})$-recurrent
configuration ${\bf u}$ such that ${\bf v}-{\bf u}\in
\langle\Delta\rangle$.

We now suppose ${\bf v}$ is a configuration. By Proposition
\ref{topple-into-stable},  we  start from ${\bf v}$, increase $v_i$
by $({\bf r}\Delta)_i$ for all $i\in\{1,2,\cdots,n\}$ and then
transform ${\bf v}+{\bf r}\Delta$ into a stable configuration by a
sequence of topplings. If we repeat the process, we shall reach
another stable configuration. This procedure can be repeated as
often as we please, whereas the number of stable configurations is
finite. So at least one of them must recur. This means that there
exists a stable configuration ${\bf u}$ for which ${\bf
u}+b\cdot{\bf r}\Delta$ can be transformed into
 ${\bf u}$
by a sequence of topplings. Hence, ${\bf u}$ is a $(\Delta,b{\bf
r})$-recurrent configuration. By Corollary \ref{parking-br-ar+br}
and Lemma \ref{lem-parking-recurrent-same}, we have ${\bf u}$ is a
$(\Delta,{\bf r})$-recurrent configuration and ${\bf u}-{\bf
v}\in\langle\Delta\rangle$.
\end{proof}

Lemma \ref{r-recurrent-coset} implies that every equivalent class of
$\mathbb{Z}^n$ contains at least one $\Delta$-recurrent
configuration. So we have the following corollary.

\begin{cor}\label{r-recurrent->=det} The number of $\Delta$-recurrent configuration is larger than or equal to $  \det \Delta$.
\end{cor}
\begin{proof} Since  the order of the quotient of
the integer  lattice $\mathbb{Z}^n/\langle \Delta\rangle$ is $\det
\Delta$, it follows from Lemma \ref{r-recurrent-coset} and Theorem
\ref{thm-r-recurrent-set-same} that $|\mathcal{R}(\Delta)|\geq \det
\Delta$.
\end{proof}

\begin{thm} $|\mathcal{P}(\Delta)|=|\mathcal{R}(\Delta)|= \det \Delta$.
\end{thm}
\begin{proof}Combining Corollaries \ref{parking-<=det},
\ref{r-recurrent->=det}, Lemma \ref{lem-parking-recurrent-same} and
Theorem \ref{thm-r-recurrent-set-same}, we have
$|\mathcal{P}(\Delta)|=|\mathcal{R}(\Delta)|= \det \Delta$.
\end{proof}

Note that recurrent configurations for a toppling matrix $\Delta$
are exactly $\Delta$-recurrent configurations. Let ${\bf
r}\in\mathscr{R}(\Delta)$. We say that a configuration ${\bf
u}=(u_1,u_2,\cdots,u_n)$ is ${\bf r}$-allowed if for any $\chi\in
\Omega(\Delta,{\bf r})$, there exists a vertex $j$ with $\chi(j)\geq
1$ such that
$$u_j\geq \Delta_{j,j}-\langle\chi,\Delta^j\rangle.$$

\begin{cor} Let ${\bf
r}\in\mathscr{R}(\Delta)$. A configuration ${\bf u}$ is a recurrent
configuration if and only if it is stable and ${\bf r}$-allowed.
\end{cor}

\end{document}